\documentclass[
12pt]{amsart}
\usepackage{amscd}
\usepackage{amssymb}
\usepackage{a4wide}
\usepackage{amstext}
\usepackage[final]{hyperref}
\hypersetup{unicode= false, colorlinks=true, linkcolor=blue,
anchorcolor=blue, citecolor=green, filecolor=red, menucolor=blue,
pagecolor=blue, urlcolor=blue} 

\newcommand{\half}{\frac{1}{2}}

\newcommand{\D}{\mathbb{D}}

 \DeclareMathOperator{\In}{(I)}
\DeclareMathOperator{\Sa}{(S)}

\newtheorem{theorem}{Theorem}
\newtheorem{lemma}{Lemma}

\theoremstyle{definition}

\title[Interpolation and sampling in small Bergman spaces]
{Interpolation and sampling in small Bergman spaces}
\author[Kristian Seip]{Kristian Seip}
\address{Department of Mathematical Sciences\\
Norwegian University of Science and Technology (NTNU)\\
 NO- 7491 Trondheim, Norway}
\email{seip@math.ntnu.no}
\thanks{The author is supported by the Research Council of
Norway grant 185359/V30. This work was started when the author
visited CMI, LATP, Universit\'{e} de Provence in the summer of 2010
and finished when he was a participant in the program \emph{Complex
Analysis and Spectral Problems} at Centre de Recerca Matem\`{a}tica,
Bellaterra in the summer of 2011. He is grateful to both
institutions for their support and hospitality.}
\begin{document}
\begin{abstract}
Carleson measures and interpolating and sampling sequences for
weighted Bergman spaces on the unit disk are described for weights
that are radial and grow faster than the standard weights
$(1-|z|)^{-\alpha}$, $0<\alpha<1$. These results make the Hardy
space $H^2$ appear naturally as a ``degenerate'' endpoint case for
the class of Bergman spaces under study.
\end{abstract}
\maketitle
\section{Introduction}
This work originates in my 1993 paper \cite{Se0} which concerns
interpolation and sampling in Bergman spaces on the unit disk with
standard radial weights $(1-|z|^2)^{-\alpha}$ and $\alpha<1$.
Following \cite{Se0}, A. Borichev, R. Dhuez, and K. Kellay
\cite{BDK} studied the same problem when the weights decay more
rapidly than any positive power of $1-|z|$ as $|z|\nearrow 1$. What
remains to be settled is then the case of nontrivial weights growing
more rapidly than $(1-|z|)^{-\alpha}$ for any $\alpha$ in $(0,1)$,
which should be thought of as dealing with Hilbert spaces of
analytic functions lying ``between" the classical Hardy and Bergman
spaces. In what follows, I will show how this can be done. Somewhat
vaguely phrased, the present analysis offers a ``smooth'' transition
from the Hardy space situation and L. Carleson's theorems (in this
context a ``degenerate'' endpoint case) and the setting of Bergman
spaces with standard weights.

Throughout this paper $w$ will be a positive and continuous function
on $[0,1)$, fixed once and for all, such that for a positive
constant $c$
\begin{equation} w(1-t)\ge c w(1-2t) \label{adm}
\end{equation} whenever $0<t\le 1/2$. We will assume that $w$ is integrable and for
convenience that
\[\int_0^1 w(x)dx =1.\]
With our fixed weight $w$, we associate the weighted Bergman space
$A^2_w$ consisting of all functions $f$ analytic in the open unit
disk $\D$ satisfying
\[ \int_{z\in \D} |f(z)|^2 w(|z|) d\sigma(z)<\infty, \]
where $\sigma$ denotes Lebesgue area measure on $\D$. The latter
integral defines a norm on $A^2_w$, but we prefer to use another
equivalent norm. Define $0\le r_n<1$ by the relation
\[ \int_{r_n}^1 w(x)dx =2^{-n}\]
for every nonnegative integer $n$, and set
\begin{equation}\label{norm} \|f\|_w^2=\sum_{n=1}^\infty 2^{-n}
\int_{0}^{2\pi} |f(r_n e^{i t})|^2 \frac{dt}{2\pi}.\end{equation}


If we had chosen to start from the sequence $(r_n)$ instead of the
weight $w$, then we would have needed to replace the condition
\eqref{adm} by the requirement that \begin{equation} \label{adm1}
\inf_{n\ge 0} \frac{1-r_n}{1-r_{n+1}}>1.\end{equation} This
alternative approach has the advantage that it permits us to
associate the Hardy space $H^2$ of the unit disk with the
``degenerate'' case when the sequence of radii $r_n$ is allowed to
be finite and $\max_n r_n=1$.

To see how the scale of Bergman spaces with standard weights fits
into this context, we introduce the following scale of weights
associated with $w$:
\[ w_{\alpha}(x)=(1-\alpha) w(x)\left(\int_{x}^1 w(t)
dt\right)^{-\alpha}
\] for $\alpha<1$. It is plain that we also have $w_{\alpha}(1-t)\ge
c_{\alpha} w_{\alpha}(1-2t)$ for some constant $c_\alpha$, and that
$\int_0^1 w_\alpha(x)dx=1$. If we choose $w\equiv 1$, then the
family of weights $w_\alpha$ corresponds to the standard weighted
Bergman spaces. Note that substituting $w$ by $w_\alpha$ corresponds
to replacing $2^{-n}$ in \eqref{norm} by $2^{-(1-\alpha)n}$. It may
be verified that this implies that the Carleson measures are
described in the same way for all the spaces $A^2_{w_\alpha}$ and
that the notion of density that we will use for $A_w^2$, also
applies to describe interpolating and sampling sequences for each of
the spaces $A^2_{w_\alpha}$.

The next sections contain two theorems, the first describing the
Carleson measures for our space $A^2_w$ and the second the
interpolating and sampling sequences for the same space. The first
theorem is easily proved using Carleson's embedding theorem, while
the second requires somewhat delicate technicalities. A main
ingredient in the proof of the second theorem is a lemma involving a
method of redistribution and atomization of certain Riesz measures.
 \vspace{1mm}

\emph{Acknowledgements.} I am grateful to Alexander Borichev for
bringing the problems studied in the present work to my attention
and for some discussions on the subject matter. I am also indebted
the anonymous referee for a careful reading and for some pertinent remarks
on the exposition of the paper. 

\section{Carleson measures for $A_w^2$}

Given a Hilbert space ${\mathcal H}$ of analytic functions on $\D$,
we say that a nonnegative Borel measure $\mu$ on $\D$ is a Carleson
measure for ${\mathcal H}$ if there exists a positive constant $C$
such that
\[ \int_{\D} |f(z)|^2 d\mu(z)\le C \|f\|_{{\mathcal H}}^2\] holds
for every $f$ in ${\mathcal H}$. The Carleson constant of $\mu$ is
the smallest possible constant $C$ for which this holds. In our
case, ${\mathcal H}$ will be either $A^2_w$ or the Hardy space
$H^2$, where the latter consists of all analytic functions $f$ in
$\D$ for which
\[ \|f\|_{H^2}=\sup_{r<1} \int_{0}^{2\pi} |f(r e^{i t})|^2
\frac{dt}{2\pi}<\infty.\]
A classical theorem of L. Carleson \cite{CaC} says that $\mu$ is a Carleson measure for 
$H^2$ if and only if there is a positive constant $C$ such that we have 
$\mu(Q_\zeta)\le C (1-|\zeta|)$ for every Carleson ``square'' $Q_\zeta=\{z: \
|\zeta|<|z|<1,\ \arg(z\overline{\zeta})<1-|\zeta|\}$, i.e., for every point $\zeta$ in $\D\setminus \{0\}$.

Before stating our theorem on Carleson measures, we introduce the
following notations, to be retained for the remainder of this paper.
Set
\[ \Omega_n=\{z:\ r_n\le |z| < r_{n+1}\}, \]
and let $\mu_n$ be the measure such that
$d\mu_n(z)=\chi_{\Omega_n}(z)d\mu(z)$ whenever a nonnegative Borel
measure $\mu$ on $\D$ is given. The notation $U(z)\lesssim V(z)$ (or
equivalently $V(z)\gtrsim U(z)$) means that there is a constant $C$
such that $U(z)\leq CV(z)$ holds for all $z$ in the set in question,
which may be a space of functions or a set of numbers. If both
$U(z)\lesssim V(z)$ and $V(z)\lesssim U(z)$, then we write
$U(z)\simeq V(z)$.
\begin{theorem} A nonnegative Borel measure $\mu$ on $\D$ is a Carleson
measure for $A^2_w$ if and only if each $\mu_n$ is a Carleson
measure for $H^2$ with Carleson constant $\lesssim 2^{-n}$.
\end{theorem}

\begin{proof}
The proof relies on Carleson's theorem \cite{CaC}. For the
necessity, it suffices to check Carleson ``squares'' $Q_\zeta=\{z: \
|\zeta|<|z|<1,\ \arg(z\overline{\zeta})<1-|\zeta|\} $ whose top
center $\zeta$ is in $\Omega_n$. We use the test function
$f_\zeta(z)=(1-\overline{\zeta}z)^{-\gamma}$ with $\gamma$ so large
that
\[ \| f_\zeta\|_w^2\simeq 2^{-n} (1-|z|)^{-2\gamma+1}; \]
this can be achieved because of \eqref{adm1}. It follows readily
from the Carleson measure condition that $\mu(Q_\zeta\cap
\Omega_n)\lesssim 2^{-n} (1-|\zeta|)$ as required by Carleson's
theorem.

To prove the sufficiency, we note that if $\mu_n$ is a Carleson
measure for $H^2$ with Carleson constant $\lesssim 2^{-n}$, then, in
view of \eqref{adm1}, the same holds for $H^2$ of the smaller disk
$r_{n+2} \D$. Given an arbitrary function $f$ in $A^2_w$, we sum the
corresponding Carleson measure estimates over $n$ and get
\[ \int_{\D} |f(z)|^2 d\mu(z) \lesssim \sum_{n=0}^{\infty} 2^{-n}
\int_{0}^{2\pi} |f(r_{n+2} e^{i t})|^2 \frac{dt}{2\pi}. \]
\end{proof}

\section{Interpolation and sampling in $A_w^2$}

Let ${\mathcal H}$ be as in the previous section, and let $K_z$ be
the reproducing kernel for ${\mathcal H}$ at the point $z$ in $\D$.
We say that a sequence $\Lambda=(\lambda_j)$ of distinct points in
$\D$ is an interpolating sequence for ${\mathcal H}$ if we can solve
the interpolation problem $f(\lambda_j)=a_j$ whenever the sequence
$(a_j)$ satisfies the admissibility condition
\[ \sum_j \frac{|a_j|^2}{K_{\lambda_j}(\lambda_j)} <\infty;\]
the sequence $\Lambda$ is said to be a sampling sequence if there
are positive constants $A$ and $B$ such that
\begin{equation}\label{frame} A\|f\|_{\mathcal H}^2\le \sum_j
\frac{|f(\lambda_j)|^2}{K_{\lambda_j}(\lambda_j)} \le B
\|f\|_{\mathcal H}^2 \end{equation} for every $f$ in ${\mathcal H}$.

We are interested in such sequences when ${\mathcal H}=A^2_w$, and
we therefore need a precise estimate for $K_z(z)$ in this case. To this end, we recall
that $K_z(z)$ is the square of the norm of the functional of point evaluation
$f\mapsto f(z)$ on $A^2_w$. By
\eqref{adm1}, we have that $1-|z|\le c (r_{n+2}-|z|)$ when $z$ is in
$\Omega_n$ for some constant $c$ independent of $n$. Thus
\begin{equation}\label{pointwise} |f(z)|^2\le C 2^n (1-|z|)^{-1} \|f\|_w^2
\end{equation} for every $f$ in $A^2_w$ and $z$ in $\Omega_n$ with $C$
independent of $n$. On the other hand, choosing $f_z$ as in the
proof of Theorem~1, we get that
\[ |f_z(z)|^2\gtrsim 2^n (1-|z|)^{-1} \|f_z\|_w^2\] if $\gamma$ is
again chosen sufficiently large; we conclude that
\begin{equation}\label{kerneldiag}
K_z(z)\simeq 2^n (1-|z|)^{-1}
\end{equation}
for $z$ in $\Omega_n$ and all $n$ when $K_z$ is the reproducing
kernel for $A^2_w$.

We denote by $\varrho(z,\zeta)$ the pseudohyperbolic distance
between two points $z$ and $\zeta$ in $\D$, i.e., \[
\varrho(z,\zeta)=\left|\frac{z-\zeta}{1-\overline{\zeta}
z}\right|.\] Let $\Lambda=(\lambda_j)$ be a separated sequence in
$\D$, which as usual we take to mean that $\inf_{j\neq
l}\varrho(\lambda_j,\lambda_l)>0$. For a given $z$, let
$n(z)=n(|z|)$ be the nonnegative integer such that $r_{n(z)}\le |z|
<r_{n(z)+1}$. We then define the following densities:
\[ D_w^{+}(\Lambda)=\limsup_{m \to
\infty}\frac{1}{m}\sup_{|z|<1}\ \sum_{|\lambda_j|\le
r_{n(z)+m}}(1-\varrho(z,\lambda_j))\] and
\[ D_w^{-}(\Lambda)=\liminf_{m \to \infty}\frac{1}{m}\inf_{|z|<1}\
\sum_{|\lambda_j|< r_{n(z)+m}}(1-\varrho(z,\lambda_j)).
\]
We future reference, we record the following consequence of 
Theorem~1.
\begin{lemma}\label{special}
The measure \[
\mu=\sum_{n=0}^\infty 2^{-n} \sum_{r_n < |\lambda_j|<r_{n+1}}
(1-|\lambda_j|)\delta_{\lambda_j}\] $\mu$ is a Carleson measure for
$A^2_w$ if and only if $D_w^+(\Lambda)<\infty$.
\end{lemma}

Our main result is the following theorem.

\begin{theorem}
$\In$ A sequence $\Lambda$ is an interpolating sequence for $A^2_w$
if and only if it is separated and $D^+_w(\Lambda)<(\log 2)/2.$
$\Sa$ A separated sequence $\Lambda$ is a sampling sequence for
$A^2_w$ if and only if $D^+_w(\Lambda)<\infty$ and
$D^-_w(\Lambda)>(\log 2)/2.$
\end{theorem}

In the ``degenerate'' case of $H^2$ (when the sequence of radii
$r_n$ is allowed to be finite and $\max_n r_n=1$), H. Shapiro and A.
Shield's $L^2$ version of Carleson's interpolation theorem \cite{SS,
Ca} gives that the condition $D^+_w(\Lambda)<(\log 2)/2$ in part (I)
should be replaced by the simpler condition  \[
\sup_{|z|<1}\sum_{j}(1-\varrho(z,\lambda_j))<\infty; \] it is
well-known that there is no counterpart to part (S) when $A_w^2$ is
replaced by $H^2$.

The densities used in Theorem~2 are defined somewhat differently
from those used in the original paper \cite{Se0} and in \cite{Se}.
These densities can also be defined via harmonic measure as shown in
\cite{OS}. It seems clear that our Theorem~2 can be rephrased in a
similar way using harmonic measure. One can of course also prove
similar results for Bergman $L^p$ spaces without any essential
changes of the arguments.

The remainder of this paper is devoted to proving respectively the
necessity (Section 4) and the sufficiency (Section 5) of the
conditions of Theorem~2.

\section{Proof of the necessity of the conditions of Theorem 2}

We begin with part (I). To see that an interpolating sequence is
separated, we can argue similarly as in the proof of Theorem~1.
Namely, if $\Lambda$ is an interpolating sequence for $A^2_w$, then
the sequence $\Lambda\cap \Omega_n$ is an interpolating sequence for
$H^2$ of the smaller disk $r_{n+2}\D$, and then, in view of
\eqref{adm1}, we may argue as we do in the classical $H^2$ case. We
omit the details of this routine argument. Lemma~\ref{special} gives that 
$D^+_w(\lambda)<\infty$ holds when $\Lambda$ is an interpolating sequence.

Just as on pages 57--58 in \cite{Se} we may modify the definition of
the upper density:
\[ D_w^{+}(\Lambda)=\limsup_{m \to
\infty}\frac{1}{m}\sup_{\lambda_l}\ \sum_{|\lambda_j|\le
r_{n(\lambda_l)+m}}(1-\varrho(\lambda_l,\lambda_j)).\] We also
repeat the argument on pages 58--59 in \cite{Se}. This means that we
only need to verify that
\[ \sum_{|\lambda_j|\le
r_{n(\lambda_l)+m}}(1-\varrho(\lambda_l,\lambda_j)) \le m (\log 2)/2
+ C \] holds with $C$ depending only on the constant of
interpolation for $\Lambda$. Assuming $\Lambda$ is an interpolating
sequence and in view of \eqref{kerneldiag}, we may solve the problem
$f_l(\lambda_{l})=2^{n(\lambda_l)/2}/\sqrt{1-|\lambda_l|}$ and
$f_l(\lambda_j)=0$ for $j\neq l$ with uniform control of norms. Let
us for simplicity set $r=r_{n(\lambda_l)+m}$. The function
$\tilde{f}_l(z)=f_l(rz)$ has $H^2$ norm $\lesssim
2^{(n(\lambda_j)+m)/2}$. Now set \[ \varphi(z)=\frac{\lambda_j
-rz}{r-\overline{\lambda_j} z}.\] We apply Jensen's formula to
$\tilde{f}_l\circ \varphi$ and get
\[(n(\lambda_j)\log 2+\log\frac{1}{1-|\lambda_l|})/2
+\sum_{j\neq l, |\lambda_j|<r}
\log\frac{|\lambda_j-\lambda_l|}{|r-\overline{\lambda_l}\lambda_j/r|}
=\int_{0}^{2\pi} \log
|\tilde{f}_l\circ\varphi(e^{it})|\frac{dt}{2\pi}.\] By the
arithmetic--geometric mean inequality,
\[ \int_{0}^{2\pi} \log
|\tilde{f}_l\circ\varphi(e^{it})|\frac{dt}{2\pi} \le
\log\|\tilde{f}_l\circ \varphi\|_{H^2}. \] Since the norm of the
composition operator is bounded by an absolute constant times $ (1-|\lambda_l|)^{-1/2}$, we obtain
\[ \int_{0}^{2\pi} \log
|\tilde{f}_l\circ\varphi(e^{it})|dt \le C + ((n(\lambda_j)+m)\log
2+\log\frac{1}{1-|\lambda_l|})/2.\] What remains is to prove that
\[\sum_{j\neq l, |\lambda_j|<r}
\log\frac{1}{|r-\overline{\lambda_l}\lambda_j/r|}\ge \sum_{j\neq l,
|\lambda_j|<r} \log\frac{1}{|1-\overline{\lambda_l}\lambda_j|}+C\]
for some constant $C$ independent of $l$ and $m$. This is a
consequence of the inequality $|1-z/a|\le |1-z|/a$ which holds
whenever $|z|\le a\le 1.$

 We turn to part (S). We obtain the condition
 $D^+_w(\Lambda)<\infty$ from the right inequality of \eqref{frame},
 cf. Lemma~\ref{special}.
Following the reasoning on page 59 of \cite{Se}, we find that it
suffices to show that $D^-_w(\Lambda)\ge (\log 2)/2$. To prove this,
we pick a point $\lambda_l$ and look at the function
\[f_{l,m}(z)=\frac{1}{1-\overline{\lambda_l}z}B_{l,m}(z),\] where $B_{l,m}$ is
the finite Blaschke product with zeros at the points $\lambda_j$ for
which $\lambda_j\neq \lambda_l$ and $|\lambda_j|\le
r_{n(\lambda_l)+m}$. By \eqref{pointwise}, we have
\[ |f_{l,m}(\lambda_l)|^2(1-|\lambda_l|) 2^{-n(\lambda_l)}\lesssim
\|f_{l,m}\|_w^2, \] which implies that \[
e^{-2m(D^{-}(\Lambda)+o(1))}(1-|\lambda_l|)^{-1}
2^{-n(\lambda_l)}\lesssim \|f_{l,m}\|_w^2\] when $m\to\infty$ and
$\lambda_l$ is chosen appropriately. We now use the fact that the
operators of multiplication by a single Blaschke factor are
uniformly bounded below on $A^2_w$. Thus applying the left
inequality of \eqref{frame} to the function \[
f_{l,m}(z)\frac{z-\lambda_l}{1-\overline{\lambda_l}z}\] and using
\eqref{kerneldiag}, we get
\[ \|f_{l,m}\|_w^2\lesssim \sum_{|\lambda_j|>r_{n(\lambda_l)+m}}
|f_{l,m}(\lambda_j)|^2(1-|\lambda_j|) 2^{-n(\lambda_j)}. \] Now applying
Theorem~1 and Carleson's embedding theorem for $H^2$, we get
\[ \|f_{l,m}\|_w^2\lesssim(1-|\lambda_l|)^{-1}
2^{-n(\lambda_l)-m},\] and the desired estimate for $D^-_w(\Lambda)$
thus follows.

\section{Proof of the sufficiency of the conditions of Theorem 2}

The main technical ingredient is the following lemma.

\begin{lemma} \label{key} Let $\Lambda$ be a separated sequence in the unit disk $\D$.
$\In$ When $D^+_w(\Lambda)<(\log 2)/2$, there are $\varepsilon>0$
and an analytic function $G(z)$ in $\D$ with zero set
$\Lambda'\supseteq \Lambda$ and
\[ |G(z)|^2\simeq 2^{(1-\varepsilon) n(z)} \varrho^2(z,\Lambda').\]
$\Sa$ When $D^-_w(\Lambda)>(\log 2)/2$ and $D^+_w(\Lambda)<\infty$,
there are $\varepsilon>0$ and a meromorphic function $G(z)$ in $\D$
with zero set $\Lambda$ and pole set $\Lambda'$ that is
pseudohyperbolically separated from $\Lambda$ and such that
\[ |G(z)|^2\simeq 2^{(1+\varepsilon) n(z)} \varrho^2(z,\Lambda)\varrho^{-2}(z,\Lambda'). \]
\end{lemma}
\begin{proof}
In either case we begin by letting $F$ be an analytic function
having $\Lambda$ as its zero set. We may write
\[ \log
|F(z)|=\sum_{n=0}^\infty (\log|B_n(z)|+h_n(z)),
\]
where $B_n$ is the Blaschke product with zeros at the $\lambda_j$ in
$\Omega_n$ and $h_n$ is an appropriate harmonic function that makes
the sum converge. The basic idea is to approximate the subharmonic
function \[ U_j(z)= \sum_{n=mj}^{mj+m-1} \log|B_n(z)|\] by another
subharmonic function $V_j(z)$ with Riesz measure supported by the
circle $|z|=r_{mj}$; the point of this redistribution of the Riesz
measure is that the latter measure is more easily atomized.

We choose $m$ so large that either (I) $-U_j(z)\le (1-\varepsilon)m$
for $|z|=r_{mj}$, where $2\varepsilon =(\log 2)/2-D^{+}_w(\Lambda)$
or (S) $-U_j(z)\ge (1+\varepsilon)m$ for $|z|=r_{mj}$, where
$2\varepsilon =D^{-}_w(\Lambda)-(\log 2)/2$. We claim that the
function
\[
V_j(z)=\frac{1}{\pi(1-r_{mj}^2)}\int_0^{2\pi}\log\left|\frac{r_{mj}e^{it}-z}
{1-\overline{z}r_{mj}e^{it}}\right| U_j(r_{mj}e^{it}) dt \] does the
job in the sense that
\[ \left|\sum_{j=1}^\infty (V_j(z)-U_j(z))\right|\le C \]
whenever $\varrho(\Lambda,z)\ge \delta>0$ with $C$ depending on
$\delta$. It is plain that we have
\[ \left|\sum_{j:\ r_{mj}\le |z|} (V_j(z)-U_j(z))\right|\le C\]
whenever $\varrho(\Lambda,z)\ge \delta>0$. To deal with the case
when $|z|<r_{mj}$, we note that then, by harmonicity, we may write
\[
U_j(z)=\int_0^{2\pi}\frac{1-|z|^2/r_{mj}^2}{|1-\overline{z}e^{it}/r_{jm}|^2}
U_j(r_{mj}e^{it}) \frac{dt}{2\pi}.\] We approximate the logarithm in
the integral defining $V_j$ as
\[-\log\left|\frac{r_{mj}e^{it}-z}
{1-\overline{z}r_{mj}}\right|=\frac{1}{2}\frac{(1-r_{mj}^2)(1-|z|^2)}
{|1-\overline{z}r_{mj}e^{it}|^2}+O\left(\left[\frac{(1-r_{mj}^2)(1-|z|^2)}
{|1-\overline{z}r_{mj}e^{it}|^2}\right]^2\right)\] when $\varrho(z,
r_{mj}e^{it})\to 1$. Here the second order term causes no problem,
so we only need to estimate the difference
\[ D_r(z)=\frac{1-|z|^2/r^2}{|1-z/r|^2}-\frac{1-|z|^2}{|1-rz|^2} \]
when $|z|<r$. It suffices to observe that
\[ D_r(z)=
\frac{1}{r^2}\left(
\frac{(1-r^2)(1-|z|^2)^2}{|1-z/r|^2|1-rz|^2}-\frac{1-r^2}{|1-z/r|^2}\right)
\]
because this identity implies that
\[ \sum_{j:r_{mj}>|z|}|D_{r_{mj}}(z)|\lesssim 1.\]

 We are now ready to construct the
desired functions $G$ in parts $\In$ and $\Sa$ resepctively. To begin with, note that we have
\[ 2^{n(z)/2}\simeq \exp\left(\sum_{j=0}^\infty \frac{\log 2}{4\pi(1-r_{mj})}
\int_0^{2\pi} \left(\log |z-r_{mj}e^{it}|-\log
r_{mj}\right)dt\right).\] In other words, the function $n(z)$ can be
approximated by a subharmonic function whose Riesz measure is
concentrated on the circles $|z|=r_{mj}$ with density $\log 2/(2
\pi(1-r_{mj})\log 2)$ with respect to arc length measure on the
respective circles. We now employ the counterpart of Lemma 5 on page
50 in \cite{Se}. In part $\In$, we thus produce an analytic function
$H(z)$ by atomizing the Riesz measure of $(1-\varepsilon)n(z)-V(z)$ for a sufficiently small $\varepsilon>0$,
and then set $G=FH$. Similarly, in part $\Sa$,  
we construct an analytic function $H$ by atomizing the Riesz measure of $V(z)-(1+\varepsilon)n(z)$ for a sufficiently small $\varepsilon>0$, and then set $G=F/H$. It is plain that the proof in \cite{Se} carries
over to this situation. (The fact that the total mass of the Riesz
measure we want to atomize over the circle $|z|=r_{mj}$ may be
non-integer is of no significance. Just leave the ``remainder''
untouched; it corresponds to a bounded part of the subharmonic
function.) 
\end{proof}

\begin{proof}[Proof of the sufficiency of the conditions in part $\In$ of Theorem~2]
We want to solve the problem $f(\lambda_j)=a_j$ with $f$ in $A^2_w$
when
\[ \sum_j |a_j|^2 (1-|\lambda_j|)2^{-n(\lambda_j)}<\infty.\]
We will in fact construct a linear operator doing the job:
\[ f(z)=\sum_j
\frac{a_j}{G'(\lambda_j)}\frac{G(z)}{z-z_j}\frac{1-|\lambda_j|^2}{1-\overline{\lambda_j}z},\]
just as formula (53) on page 53 of \cite{Se}. It follows from
Lemma~\ref{key} that
\[ |G'(\lambda_j)|\simeq
2^{(1-\varepsilon)n(\lambda_j)/2}(1-|\lambda_j|)^{-1} \ \ \
\text{and} \ \ \ \frac{|G(z)|}{|z-\lambda_j|}\lesssim
\frac{2^{(1-\varepsilon)n(z)/2}}{|1-\overline{\lambda_j}z|}\] so
that
\[|f(z)|\lesssim 2^{(1-\varepsilon)n(z)/2}\sum_j |a_j|2^{-(1-\varepsilon)n(\lambda_j)/2}
\frac{(1-|\lambda_j|)^2}{|1-\overline{\lambda_j}z|^2}. \] We write
\[h_n(z)=2^{(1-\varepsilon)n(z)/2}\sum_{j:\ r_n\le
|\lambda_j|<r_{n+1}} |a_j|2^{-(1-\varepsilon)n(\lambda_j)/2}
\frac{(1-|\lambda_j|)^2}{|1-\overline{\lambda_j}z|^2}.\] Thus we
need to show that \[
\sum_{l=1}^{\infty}2^{-l}\int_{0}^{2\pi}\left(\sum_{n=0}^\infty
h_n(r_l e^{it})\right)^2dt\lesssim \sum_j |a_j|^2
(1-|\lambda_j|)2^{-n(\lambda_j)}.\] We have \[
\sum_{l=1}^{\infty}2^{-l}\int_0^{2\pi} \left(\sum_{n=0}^\infty
h_n(r_l e^{it})\right)^2dt\le I_1 +I_2,\] where
\[
I_1=\sum_{l=1}^{\infty}2^{-l+1}\int_0^{2\pi}\left(\sum_{n<l} h_n(r_l
e^{it})\right)^2 dt \ \ \ \text{and} \ \
 \ I_2=\sum_{l=1}^{\infty}2^{-l+1}\int_0^{2\pi}\left(\sum_{n\ge l} h_n(r_l
e^{it})\right)^2 dt.\] We compute the $L^2$ integral in $I_1$ by
duality. Noting that $h_n(z)$ is a weighted sum of Poisson kernels and using the Carleson measure condition
for Poisson integrals of $L^2$ functions along with the Cauchy--Schwarz inequality, we then get
\[ I_1 \lesssim \sum_{l=1}^{\infty}2^{-\varepsilon l} \left(\sum_{n<l}
2^{\varepsilon n/2} \left(\sum_{j:\ r_n \le |\lambda_j|<r_{n+1}}
|a_j|^2 (1-|\lambda_j|)2^{-n(\lambda_j)}\right)^{\half}\right)^2.\]
By the Cauchy--Schwarz inequality, we get
\[ I_1 \lesssim \sum_{l=1}^{\infty}2^{-\varepsilon l/2} \sum_{n<l}
2^{\varepsilon n/2} \sum_{j:\ r_n \le |\lambda_j|<r_{n+1}} |a_j|^2
(1-|\lambda_j|)2^{-n(\lambda_j)},\] and the desired estimate is
obtained if we change the order of summation. To deal with $I_2$, we
note that \[ \left(\sum_{n\ge l} h_n(z)\right)^2  \le
2^{(1-\varepsilon)n(z)} \sum_{|\lambda_j|\ge r_l} |a_j|^2
2^{-(1-\varepsilon)
n(\lambda_j)}\frac{(1-|\lambda_j|)^2}{|1-\overline{\lambda_j}z|^{2}}
\sum_{|\lambda_k|\ge r_l}
\frac{(1-|\lambda_k|)^2}{|1-\overline{\lambda_k} z|^{2}}.\] Thus
\[ \left(\sum_{n\ge l} h_n(r_l e^{it})\right)^2  \lesssim
2^{(1-\varepsilon) l} \sum_{|\lambda_j|\ge r_l} |a_j|^2
2^{-(1-\varepsilon)
n(\lambda_j)}\frac{(1-|\lambda_j|)^2}{|1-\overline{\lambda_j}r_le^{ir_l}|^{2}}
 \]
from which it follows that
\[ \int_0^{2\pi}\left(\sum_{n\ge l} h_n(r_l e^{it})\right)^2dt
\lesssim 2^{(1-\varepsilon)l} \sum_{|\lambda_j|\ge r_l} |a_j|^2
2^{-(1-\varepsilon)
n(\lambda_j)}\frac{(1-|\lambda_j|)^{2}}{(1-r_l)}.\] Finally, we get
\[ I_2\lesssim \sum_j |a_j|^2
2^{-(1-\varepsilon) n(\lambda_j)}(1-|\lambda_j|)^{2} \sum_{r_l\le
|\lambda_j|}\frac{2^{-\varepsilon l}} {(1-r_l)}.\] In the latter
sum, we can assume that $\varepsilon$ is so small (if need be) that
the terms grow exponentially, so that we may arrive at the desired
estimate.
\end{proof}

\begin{proof}[Proof of the sufficiency of the conditions in part $\Sa$ of Theorem~2]
We start from the formula
\[ f(z)=\sum_j
\frac{f(\lambda_j)}{G'(\lambda_j)}\frac{G(z)}{z-z_j}
\frac{1-|z|^2}{1-\overline{z}\lambda_j},\] which holds for every
function in $A_w^2$, cf. formula (54) on page 53 of \cite{Se}. Here
$G$ is the meromorphic function of Lemma~\ref{key}. We get that
\[|f(z)|\lesssim 2^{(1+\varepsilon)n(z)/2}(1-|z|)
\sum_j |f(\lambda_j)|2^{-(1+\varepsilon)n(\lambda_j)/2}
\frac{(1-|\lambda_j|)}{|1-\overline{\lambda_j}z|^2}. \] We write
\[g_n(z)=2^{(1+\varepsilon)n(z)/2}(1-|z|)\sum_{j:\ r_n\le
|\lambda_j|<r_{n+1}}
|f(\lambda_j)|2^{-(1+\varepsilon)n(\lambda_j)/2}
\frac{(1-|\lambda_j|)}{|1-\overline{\lambda_j}z|^2}.\] Thus we need
to show that \[
\sum_{l=1}^{\infty}2^{-l}\int_{0}^{2\pi}\left(\sum_{n=0}^\infty
g_n(r_l e^{it})\right)^2dt\lesssim \sum_j |f(\lambda_j)|^2
(1-|\lambda_j|)2^{-n(\lambda_j)}.\] We write \[
\sum_{l=1}^{\infty}2^{-l}\int_0^{2\pi} \left(\sum_{n=0}^\infty
g_n(r_l e^{it})\right)^2 dt \le J_1 +J_2,\] where
\[
J_1=\sum_{l=1}^{\infty}2^{-l+1}\int_0^{2\pi}\left(\sum_{n<l} g_n(r_l
e^{it})\right)^2 dt \ \ \ \text{and} \ \
 \ J_2=\sum_{l=1}^{\infty}2^{-l+1}\int_0^{2\pi}\left(\sum_{n\ge l} g_n(r_l
e^{it})\right)^2 dt.\] We compute the $L^2$ integral in $J_1$ by
duality. Using the Carleson measure condition and the
Cauchy--Schwarz inequality, we get
\[ J_1 \lesssim \sum_{l=1}^{\infty} (1-r_l)^{\half}2^{\varepsilon l} \sum_{n<l}
\frac{2^{-\varepsilon n}}{(1-r_n)^{\half}} \sum_{j:\ r_n \le
|\lambda_j|<r_{n+1}} |f(\lambda_j)|^2
(1-|\lambda_j|)2^{-n(\lambda_j)}.\] Changing the order of summation,
we get the desired result. (We need to assume, if need be, that
$\varepsilon$ is so small that $2^{\varepsilon l}(1-r_l)^{\half}$
decays exponentially.) To deal with $J_2$, we note that
\[ \left(\sum_{n\ge l} g_n(z)\right)^2 \le
2^{(1+\varepsilon)l}(1-|z|)^2 \sum_{|\lambda_j|\ge r_l}
|f(\lambda_j)|^2 2^{-(1+\half\varepsilon)
n(\lambda_j)}\frac{(1-|\lambda_j|)}{|1-\overline{\lambda_j}z|^{2}}
\sum_{|\lambda_k|\ge r_l}
\frac{(1-|\lambda_k|)}{|1-\overline{\lambda_k} z|^{2}}2^{-\half
\varepsilon n(\lambda_k)}.\] Applying the Carleson measure condition
to the sum to the right, we get
\[ \left(\sum_{n\ge l} g_n(r_l e^{it})\right)^2  \le
2^{l(1+\half\varepsilon)} \sum_{|\lambda_j|\ge r_l}
|f(\lambda_j)|^2(1-|\lambda_j|) 2^{-(1+\half\varepsilon)
n(\lambda_j)}\frac{(1-r_l)}{|1-\overline{\lambda_j}r_l e^{it}|^{2}}
\] from which it follows that
\[ \int_0^{2\pi}\left(\sum_{n\ge l} g_n(r_l e^{it})\right)^2dt
\lesssim 2^{\varepsilon l/2} \sum_{|\lambda_j|\ge r_l}
|f(\lambda_j|^2 (1-|\lambda_j|)2^{-(1+\half\varepsilon)
n(\lambda_j)}.\] We now sum over $l$ and get the desired result by
changing the order of summation.
\end{proof}


\begin{thebibliography}{BRSHZE}


\bibitem{BDK} A. Borichev, R. Dhuez, and K. Kellay, \emph{Sampling and
interpolation in large Bergman and Fock spaces}, J. Funct. Anal.
\textbf{242} (2007), 563–-606.


\bibitem{Ca}L. Carleson, \emph{An interpolation problem for bounded
analytic functions}, Amer. J. Math. \textbf{80} (1958), 921–-930.

\bibitem{CaC}L. Carleson, \emph{Interpolations by bounded analytic
functions and the corona problem} Ann. of Math. (2) \textbf{76}
(1962), 547-–559.




\bibitem{OS}J. Ortega-Cerd\`{a} and K. Seip, \emph{Harmonic measure
and uniform densities}, Indiana Univ. Math. J. \textbf{53} (2004),
905–-923.

\bibitem{Se0}K. Seip, \emph{Beurling type density theorems in the
unit disk}, Invent. Math. \textbf{113} (1993), 21--39.

\bibitem{Se} K. Seip, \emph {Interpolation and Sampling in Spaces of
Analytic Functions}, University Lecture Series \textbf{33}, Amer.
Math. Soc., Providence, RI, 2004.  xii+139 pp. ISBN: 0-8218-3554-8.

\bibitem{SS}H.~S. Shapiro and A. L. Shields, \emph{On some interpolation
problems for analytic functions}, Amer. J. Math. \textbf{83} (1961),
513-–532.

\end{thebibliography}
\end{document}